\documentclass[12pt]{article}

\usepackage[margin=1in]{geometry}
\usepackage{amsmath, amssymb, amsthm}
\usepackage{hyperref}

\theoremstyle{plain}
\newtheorem{theorem}{Theorem}[section]
\newtheorem{lemma}[theorem]{Lemma}
\newtheorem{proposition}[theorem]{Proposition}
\theoremstyle{definition}
\newtheorem{remark}[theorem]{Remark}

\newcommand{\Z}{\mathbb{Z}}

\newcommand{\stirling}[2]{\genfrac\{\}{0pt}{}{#1}{#2}}
\newcommand{\vp}{v_p}

\title{A Proof of Bala's Congruence Conjectures for A158690}
\author{Ahaan Kallat\\
Khoury College of Computer Sciences\\
Northeastern University\\
Boston, MA, USA\\
\texttt{kallat.a@northeastern.edu}}
\date{July 25, 2026}

\begin{document}
\maketitle

\begin{abstract}
Let $a(n)$ be the sequence A158690 in the On-Line Encyclopedia of Integer
Sequences (OEIS), defined by the exponential generating function
$\sum_{n\ge0}a(n)t^n/n!=1+\sum_{m\ge1}\prod_{j=1}^{m}(1-e^{-(2j-1)t})$.
We prove two congruence conjectures of Peter Bala. The first states that, for
every integer $k\ge1$, the sequence $a(n)$ modulo $k$ is eventually periodic
with period dividing $\varphi(k)$; we prove the stronger statement that the
Carmichael function $\lambda(k)$ is an eventual period. The second asserts
that $a(np^r+i)\equiv a(np^{r-1}+i)\pmod{p^r}$ for every $i\ge0$, every
prime $p$, and all $n,r\ge1$. The proof places A158690 in a general class of
exponential generating functions of the form $G(e^t-1)$ with
$G\in\mathbb{Z}[[y]]$, expands the coefficients in the Stirling-number basis,
and reduces both conclusions to elementary congruences for integer powers.
\end{abstract}



\section{Introduction}

The On-Line Encyclopedia of Integer Sequences records many empirical
observations contributed by its users and labeled ``Conjecture.'' This note
proves two such conjectures, attributed to Peter Bala, for the sequence
A158690 \cite{oeis-a158690}.

The sequence begins
\[
 1,1,5,55,1073,32671,1431665,85363615,\ldots
\]
and is defined by the exponential generating function
\begin{equation}
 F(t):=\sum_{n\ge0}a(n)\frac{t^n}{n!}
 =1+\sum_{m\ge1}\prod_{j=1}^{m}
   \bigl(1-e^{-(2j-1)t}\bigr).
 \label{eq:A158690-egf}
\end{equation}
The sequence occurs among sum-of-finite-product generating functions related
to Fishburn-type structures; see, for example, Hwang and Jin
\cite{hwang-jin}. The OEIS entry also records connections with $q$-series
studied by Zagier \cite{zagier}.

All generating-function identities in this note are identities of formal
power series. In particular, no analytic convergence is needed. In
\eqref{eq:A158690-egf}, the product indexed by $m$ has order at least $m$ at
$t=0$, so every coefficient receives contributions from only finitely many
products.

Bala formulated the following two conjectures on December 18, 2021
\cite{oeis-a158690}.

\begin{enumerate}
\item For every integer $k\ge1$, the sequence $a(n)$ modulo $k$ is eventually
periodic with period dividing $\varphi(k)$.

\item For $i\ge0$, set $a_i(n)=a(n+i)$. For every prime $p$ and all positive
integers $n,r$,
\begin{equation}
 a_i(np^r)\equiv a_i(np^{r-1})\pmod{p^r}.
 \label{eq:bala-gauss}
\end{equation}
\end{enumerate}

We prove both statements through a general observation. Suppose
$G(y)\in\Z[[y]]$ and define $b(n)$ by
\begin{equation}
 G(e^t-1)=\sum_{n\ge0}b(n)\frac{t^n}{n!}.
 \label{eq:general-egf}
\end{equation}
Then the coefficients $b(n)$ are integral linear combinations of the
sequences $m!\stirling{n}{m}$. The latter admit finite power-sum expansions,
and elementary congruences for integer powers give both periodicity and the
Gauss congruences. Differentiation preserves the form \eqref{eq:general-egf},
which supplies all shifts at once.

The main result is the following.

\begin{theorem}\label{thm:main}
The two conjectures above hold for A158690. More precisely:
\begin{enumerate}
\item for every $k\ge2$,
\[
 a(n+\lambda(k))\equiv a(n)\pmod{k}
\]
for all sufficiently large $n$, where $\lambda$ is the Carmichael function;

\item for every $i\ge0$, every prime $p$, and all $n,r\ge1$,
\[
 a(np^r+i)\equiv a(np^{r-1}+i)\pmod{p^r}.
\]
\end{enumerate}
Since $\lambda(k)\mid\varphi(k)$, the first statement implies Bala's claimed
period bound.
\end{theorem}

The proof has two steps. Section~\ref{sec:general} proves the congruence
theorem for every series $G(e^t-1)$ with $G\in\Z[[y]]$. Section~\ref{sec:application}
shows directly that the generating function of A158690 has this form.

\section{A general congruence theorem}\label{sec:general}

Let
\[
 G(y)=\sum_{m\ge0}c_m y^m\in\Z[[y]],
\]
and define $b(n)$ by \eqref{eq:general-egf}. We use the standard Stirling
number identity
\begin{equation}
 (e^t-1)^m
 =m!\sum_{n\ge m}\stirling{n}{m}\frac{t^n}{n!},
 \label{eq:stirling-egf}
\end{equation}
where $\stirling{n}{m}$ denotes a Stirling number of the second kind. It
follows that
\begin{equation}
 b(n)=\sum_{m=0}^{n}c_m m!\stirling{n}{m}.
 \label{eq:b-stirling}
\end{equation}
We will also use the inclusion--exclusion formula
\begin{equation}
 m!\stirling{N}{m}
 =\sum_{j=0}^{m}(-1)^{m-j}\binom{m}{j}j^N,
 \qquad N\ge1.
 \label{eq:stirling-powers}
\end{equation}
Both formulas are standard; see, for example, \cite[Chapter~6]{concrete}.

\subsection{Eventual periodicity}

\begin{proposition}\label{prop:periodicity}
Let $G(y)\in\Z[[y]]$, and let $b(n)$ be defined by
\eqref{eq:general-egf}. For every $k\ge2$, $\lambda(k)$ is an eventual period
of $b(n)$ modulo $k$. Since $\lambda(k)\mid\varphi(k)$, this implies that $b(n)$ has an
eventual period dividing $\varphi(k)$.
\end{proposition}

\begin{proof}
Fix $k\ge2$. Choose $M$ such that $k\mid M!$; for instance, $M=k$ is enough.
For every $m\ge M$, the integer $m!\stirling{n}{m}$ is divisible by $k$.
Therefore \eqref{eq:b-stirling} gives
\begin{equation}
 b(n)\equiv
 \sum_{m=0}^{M-1}c_m m!\stirling{n}{m}
 \pmod{k}.
 \label{eq:finite-truncation}
\end{equation}
Thus, modulo $k$, the sequence $b(n)$ is a fixed finite integral linear
combination of the power sequences $j^n$, by
\eqref{eq:stirling-powers}.

Write
\[
 k=\prod_{p^\alpha\parallel k}p^\alpha
\]
and put
\[
 N_0=\max_{p^\alpha\parallel k}\alpha.
\]
Fix one of the finitely many integers $j$ occurring in
\eqref{eq:finite-truncation}. We claim that for every $n\ge N_0$,
\begin{equation}
 j^{n+\lambda(k)}\equiv j^n\pmod{k}.
 \label{eq:power-period}
\end{equation}
It is enough to verify the congruence modulo every $p^\alpha\parallel k$.
If $p\nmid j$, then $\lambda(p^\alpha)\mid\lambda(k)$, so the definition of
the Carmichael function gives
\[
 j^{\lambda(k)}\equiv1\pmod{p^\alpha}.
\]
If $p\mid j$, then $n\ge\alpha$ implies
\[
 j^n\equiv j^{n+\lambda(k)}\equiv0\pmod{p^\alpha}.
\]
This proves \eqref{eq:power-period} by the Chinese remainder theorem.
Applying it to every power sequence in the fixed finite sum
\eqref{eq:finite-truncation} yields
\[
 b(n+\lambda(k))\equiv b(n)\pmod{k}
\]
for every $n\ge N_0$. Finally, $\lambda(k)\mid\varphi(k)$, so Bala's stated
period bound follows.
\end{proof}

\begin{remark}
The proof gives the explicit, though not necessarily sharp, preperiod bound
\[
 n\ge \max_{p^\alpha\parallel k}\alpha.
\]
For $k=16$, this gives $n\ge4$. Bala's example on the OEIS entry has exactly
four initial terms before the displayed period begins.
\end{remark}

\subsection{Gauss congruences}

The second conjecture rests on the following elementary congruence for powers.

\begin{lemma}\label{lem:power-gauss}
Let $p$ be prime and let $j,n,r$ be integers with $j\ge0$ and $n,r\ge1$.
Then
\begin{equation}
 j^{np^r}\equiv j^{np^{r-1}}\pmod{p^r}.
 \label{eq:power-gauss}
\end{equation}
\end{lemma}

\begin{proof}
If $j=0$, then both sides of \eqref{eq:power-gauss} are $0$, since
$n,r\ge1$ and hence both exponents are positive. Thus assume $j\ge1$.

If $p\nmid j$, then
\[
 np^r-np^{r-1}=np^{r-1}(p-1)=n\varphi(p^r),
\]
so Euler's theorem gives \eqref{eq:power-gauss}.

Now suppose $p\mid j$. Since $j\ge1$ and $n\ge1$,
\[
 \vp\!\left(j^{np^{r-1}}\right)
 =np^{r-1}\vp(j)\ge p^{r-1}\ge r.
\]
The last inequality holds for every prime $p$ and $r\ge1$. Hence
$j^{np^{r-1}}\equiv0\pmod{p^r}$, and the same is true of $j^{np^r}$.
\end{proof}

\begin{proposition}\label{prop:gauss}
Let $G(y)\in\Z[[y]]$, and let $b(n)$ be defined by
\eqref{eq:general-egf}. Then for every prime $p$ and all $n,r\ge1$,
\begin{equation}
 b(np^r)\equiv b(np^{r-1})\pmod{p^r}.
 \label{eq:general-gauss}
\end{equation}
\end{proposition}

\begin{proof}
By Lemma~\ref{lem:power-gauss}, each term in
\eqref{eq:stirling-powers} satisfies
\[
 j^{np^r}\equiv j^{np^{r-1}}\pmod{p^r}.
\]
Therefore, for every $m\ge0$,
\begin{equation}
 m!\stirling{np^r}{m}
 \equiv
 m!\stirling{np^{r-1}}{m}
 \pmod{p^r}.
 \label{eq:stirling-gauss}
\end{equation}
Here we use the usual convention $\stirling{N}{m}=0$ when $m>N$, so the
congruence is valid even when $np^{r-1}<m\le np^r$.

Using \eqref{eq:b-stirling} on both sides and extending the shorter sum by
zero terms,
\begin{align*}
 b(np^r)-b(np^{r-1})
 &=\sum_{m=0}^{np^r} c_m m!
 \left(
 \stirling{np^r}{m}-\stirling{np^{r-1}}{m}
 \right),
\end{align*}
which is divisible by $p^r$ term by term by
\eqref{eq:stirling-gauss}. This proves \eqref{eq:general-gauss}.
\end{proof}

It remains to explain why every shift of $b(n)$ is again covered by the same
proposition.

\begin{proposition}\label{prop:shifts}
Under the hypotheses of Proposition~\ref{prop:gauss}, fix $i\ge0$ and set
$b_i(n)=b(n+i)$. Then there exists $G_i(y)\in\Z[[y]]$ such that
\[
 G_i(e^t-1)=\sum_{n\ge0}b_i(n)\frac{t^n}{n!}.
\]
Consequently, for every prime $p$ and all $n,r\ge1$,
\[
 b(np^r+i)\equiv b(np^{r-1}+i)\pmod{p^r}.
\]
\end{proposition}

\begin{proof}
Differentiating the exponential generating function $i$ times gives
\[
 \frac{d^i}{dt^i}G(e^t-1)
 =\sum_{n\ge0}b(n+i)\frac{t^n}{n!}.
\]
Put $y=e^t-1$. The formal chain rule gives
\[
 \frac{d}{dt}=(1+y)\frac{d}{dy}.
\]
Hence, with the operator
\[
 D=(1+y)\frac{d}{dy},
\]
we have
\[
 \frac{d^i}{dt^i}G(e^t-1)
 =(D^iG)(e^t-1).
\]
Differentiation maps $\Z[[y]]$ to $\Z[[y]]$, as does multiplication by
$1+y$, so
\[
 G_i(y):=D^iG(y)\in\Z[[y]].
\]
Proposition~\ref{prop:gauss}, applied to $G_i$, proves the shifted
congruence.
\end{proof}

Combining Propositions~\ref{prop:periodicity}, \ref{prop:gauss}, and
\ref{prop:shifts} gives a general theorem that may be useful independently of
A158690.

\begin{theorem}\label{thm:general}
Let $G(y)\in\Z[[y]]$ and define $b(n)$ by
\[
 G(e^t-1)=\sum_{n\ge0}b(n)\frac{t^n}{n!}.
\]
Then:
\begin{enumerate}
\item for every $k\ge2$, $\lambda(k)$ is an eventual period of $b(n)$ modulo
$k$;
\item for every $i\ge0$, every prime $p$, and all $n,r\ge1$,
\[
 b(np^r+i)\equiv b(np^{r-1}+i)\pmod{p^r}.
\]
\end{enumerate}
\end{theorem}

\section{Application to A158690}\label{sec:application}

We now verify that the generating function \eqref{eq:A158690-egf} is of the
form required by Theorem~\ref{thm:general}.

Set
\[
 y=e^t-1,
\qquad\text{so that}\qquad
 e^{-t}=(1+y)^{-1}.
\]
For a positive integer $d$,
\begin{equation}
 (1+y)^{-d}
 =\sum_{r\ge0}(-1)^r\binom{d+r-1}{r}y^r
 \in\Z[[y]].
 \label{eq:negative-binomial}
\end{equation}
In particular,
\[
 1-(1+y)^{-d}\in y\Z[[y]].
\]
Define
\begin{equation}
 H(y)
 =1+\sum_{m\ge1}\prod_{j=1}^{m}
 \left(1-(1+y)^{-(2j-1)}\right).
 \label{eq:H}
\end{equation}
The $m$th product in \eqref{eq:H} belongs to $y^m\Z[[y]]$. Therefore the
infinite sum is well defined coefficientwise: the coefficient of $y^N$
receives contributions only from $m\le N$. Equation
\eqref{eq:negative-binomial} also shows that every coefficient is an integer.
Thus
\[
 H(y)\in\Z[[y]].
\]
Finally, substituting $y=e^t-1$ into \eqref{eq:H} gives
\begin{align*}
 H(e^t-1)
 &=1+\sum_{m\ge1}\prod_{j=1}^{m}
 \left(1-e^{-(2j-1)t}\right)\\
 &=F(t).
\end{align*}
Hence the sequence A158690 satisfies the hypotheses of
Theorem~\ref{thm:general}. The first conclusion proves Bala's eventual
periodicity conjecture, with the stronger eventual period $\lambda(k)$, and
the second conclusion proves all of Bala's shifted Gauss congruences. This
proves Theorem~\ref{thm:main}.

\section{Concluding remarks}

The proof uses very little of the particular product defining A158690. Its
role is only to show that the exponential generating function can be written
as $G(e^t-1)$ for some $G\in\Z[[y]]$, in the sense of
Theorem~\ref{thm:general}. Once that representation is available, the two
apparently different conjectures come from the same Stirling-number basis.

For eventual periodicity, the factorial factor in
$m!\stirling{n}{m}$ makes all sufficiently large $m$ vanish modulo a fixed
modulus, leaving a finite combination of power sequences. For the Gauss
congruences, the same power sequences satisfy the required prime-power
congruence term by term. The fact that series of the form $G(e^t-1)$ with $G\in\Z[[y]]$ remain of
this form under differentiation then explains why every shift appears in
Bala's second conjecture.

Theorem~\ref{thm:general} therefore applies without change to any other
integer sequence whose exponential generating function can be written as
$G(e^t-1)$ with $G\in\Z[[y]]$.

\section*{Acknowledgments}

The author thanks the contributors to the OEIS entry for A158690, especially
Peter Bala for formulating the two congruence conjectures. The author used AI
tools for exploratory discussion, computational checking, drafting assistance,
and assistance with the Lean formalization. The author independently checked
the mathematical arguments, computations, and final manuscript. Bala's two
original congruence conjectures have additionally been formally verified in the
Lean~4 proof assistant using Mathlib.\footnote{Formalization available at
\url{https://github.com/ahaankallat/a158690-lean}. The top-level Lean statements
are \texttt{A158690.bala\_conjecture\_one} and
\texttt{A158690.bala\_conjecture\_two}. The Lean development verifies Bala's
original $\varphi(k)$ period statement; the stronger $\lambda(k)$ refinement
proved in Proposition~\ref{prop:periodicity} is not needed for the conjecture.}

\end{document}